\documentclass{article}
\pdfoutput=1
\usepackage{amsthm}
\usepackage{amsmath}
\usepackage{amssymb}
\usepackage{amscd}
\usepackage{graphicx}

\newtheorem{theorem}[subsection]{Theorem}
\newtheorem{lemma}[subsection]{Lemma}

\theoremstyle{definition}

\theoremstyle{remark}

\newtheorem{remark}[subsection]{Remark}

\vspace{1.5cm}
\title{{\bf Newton flows for elliptic functions III}\\
{\bf {\small Classification of 
$3^{\text{rd}}$ order Newton graphs}}}

\author{G.F. Helminck,\\
Korteweg-de Vries Institute\\
University of Amsterdam\\
P.O. Box 94248\\
1090 GE Amsterdam\\
The Netherlands\\
e-mail: g.f.helminck@uva.nl\\ 
F. Twilt,\\
Department of Applied Mathematics\\
University of Twente\\
P.O. Box 217, 7500 AE Enschede\\
The Netherlands\\
e-mail: f.twilt@kpnmail.nl\\
}

\begin{document}
\maketitle

\begin{abstract}
\noindent
A Newton graph of order $r( \geqslant 2)$ is a cellularly embedded toroidal graph on $r$ vertices, $2r$ edges and $r$ faces that fulfils certain combinatorial properties (Euler, Hall). The significance of these graphs relies on their role in the study of {\it structurally stable elliptic Newton flows} - say $\overline{\overline{\mathcal{N}}} (f)$ - of order $r$, i.e. desingularized continuous versions of Newton's iteration method for finding zeros for an {\it elliptic} function $f$ (of order $r$). In previous work we established a representation of these flows in terms of Newton graphs. The present paper results into the classification of all $3^{\text{rd}}$ order Newton graphs, implying a list of all nine possible $3^{\text{rd}}$ order flows $\overline{\overline{\mathcal{N}}} (f)$ (up to conjugacy and duality). 
\end{abstract}



\section{Motivation and preliminaries}

\noindent
{\large{\bf {\small 1.1 Newton flows vs. Newton graphs}}}\\

 Throughout this paper the connected graph $\mathcal{G}$ is a cellular\footnote{\label{Vtnt1} i.e., each  face is homeomorphic to an open $\mathbb{R}^{2}$-disk.} embedding in the torus $T$ of an abstract connected multigraph $\mathcal{G}$ (i.e., no loops) with $r$ vertices, 2$r$ edges ($r \geqslant 2$) and thus $r$ faces; $r=$order $\mathcal{G}$.
  We say that $\mathcal{G}$ has the {\it A(angle)-property} if all angles at a vertex in the boundary of a face spanning a sector of this face, are well defined, strictly positive and sum up to 2$\pi$. The {\it A-property} has a combinatorial interpretation (Hall), cf. \cite{HT2}.
  We say that $\mathcal{G}$ has the {\it E(Euler)-property} if the boundary of each face, as subgraph of $\mathcal{G}$, is Eulerian, i.e., admits a closed facial walk that traverses each edge only once and goes through all vertices. 
  The graph $\mathcal{G}$ is called a {\it Newton graph} if both the {\it A-property} and the {\it E-property} hold.
  
    It is proved (\cite{HT2}) that the geometrical dual (denoted $\mathcal{G}^{*}$) of a Newton graph $\mathcal{G}$ is also Newtonian.
The anti-clockwise permutation on the embedded edges at vertices of $\mathcal{G}$ induces a clockwise orientation of the facial walks on the boundaries of the $\mathcal{G}$-faces, cf. Fig.1-(a). On its turn, the clockwise orientation of $\mathcal{G}$-faces gives rise to a clockwise permutation on the embedded edges at the vertices of $\mathcal{G}^{*}$, and thus to an anti-clockwise orientation of $\mathcal{G}^{*}$. In the sequel $\mathcal{G}$ and $\mathcal{G}^{*}$ are always oriented in this way: $\mathcal{G}$ clockwise ($-$), $\mathcal{G}^{*}$ anti-clockwise ($+$). Altogether, we find:  $(\mathcal{G}^{*})^{*}=\mathcal{G}$.

The significance of Newton graphs relies on the study of so called {\it elliptic Newton flows}:\\
With $f$ a non-constant elliptic (i.e., meromorphic, doubly periodic) function of order $r$ ($\geqslant 2$), we considered (\cite{HT1},\cite{HT2}) 
$C^{1}$-vector fields (flows), denoted $\overline{\overline{\mathcal{N}}} (f)$, on $T$ that are defined Ð on each chart of $T$Ð as a toroidal, desingularized version of the planar dynamical system\footnote{\label{vtnt2} In fact, we considered the system $  \dfrac{dz}{dt} =-(1+|f(z)|^{4})^{-1}|f'(z)|^{2}\dfrac{f (z)}{f^{'} (z)}$: a continuous version of Newton's damped iteration method for finding zeros for $f$, see \cite{HT1}.} given by
\begin{equation}
\label{vgl2x}
  \dfrac{dz}{dt} = \dfrac{-f (z)}{f^{'} (z)}, z \in \mathbb{C},
\end{equation}
thereby focussing on qualitative features of phase portraits (families of trajectories). Here, zeros, poles and critical points [i.e., $f'$ vanishes but $f$ not] of $f$ serve as resp. attractors, repellors and saddles.  We emphasize that $\overline{\overline{\mathcal{N}}} (f)$ is 
{\it not} complex analytic.

The flow $\overline{\overline{\mathcal{N}}} (f)$ is called {\it structural stable} if its phase portrait is topologically invariant under small perturbations of the zeros and poles for $f$. We obtained: (cf. \cite{HT1}, \cite{Peix2}) \\
{\bf Characterization:} $\overline{\overline{\mathcal{N}}} (f)$ is structurally stable iff there holds:\\
(i) all  zeros,  poles and critical points for $f$ are simple,\\
(ii) the phase portrait does not exhibit ``saddle connections''.\\
{\bf Genericity:}\;\;\;\;\;\;\;\;\;\;\;\;\;$\overline{\overline{\mathcal{N}}} (f)$ is structurally stable for ``almost all''\footnote{\label{Vtnt3} $f$ in an open and dense set of the set of all functions $f$ of order $r$. (w.r.t. an appropriate topology)} functions $f$.\\
{\bf Duality:}\;\;\;\;\;\;\;\;\;\;\;\;\;\;\;\;\;\;If  $\overline{\overline{\mathcal{N}}} (f)$ is structurally stable, 
then also $\overline{\overline{\mathcal{N}}} (\frac{1}{f})$ and  $\overline{\overline{\mathcal{N}}} (\frac{1}{f})=-\overline{\overline{\mathcal{N}}} (f)$.\\
  Let $\overline{\overline{\mathcal{N}}} (f)$ be structurally stable, then $\mathcal{G}(f )$ is a toroidal graph with as vertices the attractors, as edges the unstable manifolds at saddles and as faces the basins of repulsion of the repellors for $\overline{\overline{\mathcal{N}}} (f)$. It turns out that $\mathcal{G}(f )$ is a Newton graph of order $r$ endowed with the clockwise orientation and  moreover, $\mathcal{G}(\frac{1}{f})$=$-\mathcal{G}(f)^{*}$. 
  
The main result obtained in \cite{HT2} is:\\
{\bf Representation and classification:} (all graphs and flows of order $r$)\\
Given a Newton graph $\mathcal{G}$,  a structurally stable flow $\overline{\overline{\mathcal{N}}} (f_{\mathcal{G}})$ exists such that:
$$
\mathcal{G}(f_{\mathcal{G}}) \sim \mathcal{G}, (\text{ thus }\mathcal{G}^{*} \sim - \mathcal{G}(\frac{1}{f_{\mathcal{G}}}))
$$
   and,  if $\mathcal{G}, \mathcal{H}$ are Newton graphs, then:
 \begin{equation}
\label{rep}
\overline{\overline{\mathcal{N}}} (f_{\mathcal{G}}) \sim \overline{\overline{\mathcal{N}}} (f_{\mathcal{H}}) \Leftrightarrow \mathcal{G} \sim \mathcal{H},
\end{equation}
where, $\sim$ in the l.h.s. stands for conjugacy\footnote{\label{Vtn5} Two elliptic Newton flows are conjugate if an homeomorphism from $T$ onto itself exists mapping the phase portrait of one flow
onto that of the other, thereby respecting the orientations of the trajectories.} between Newton flows, and $\sim$ in the r.h.s. for equivalency (i.e., an orientation preserving isomorphism\footnote{\label{Vtn6} i.e., 
between the underlying abstract graphs, respecting the oriented faces of the embedded graphs.})
between Newton graphs.\\
 $\mathcal{G}(f )$ is, so to say,  the principal part of the phase portrait of the structurally stable flow $\overline{\overline{\mathcal{N}}} (f)$ and determines, in a qualitative sense, the whole phase portrait; see Fig. 1-(a), (b) for an illustration.
  In accordance with our philosophy (``focus on qualitative aspects''), conjugate flows are considered as equal. Note however, that by the above classification we have: $ \overline{\overline{\mathcal{N}}} (f) \sim \overline{\overline{\mathcal{N}}} (\frac{1}{f})$ iff   $\mathcal{G}(f ) \sim -\mathcal{G}(f )^{*}$, which is in general not true\footnote{\label{Vtnt6}If $\overline{\overline{\mathcal{N}}} (f)\sim \overline{\overline{\mathcal{N}}} (\frac{1}{f})$, and thus  $\mathcal{G}(f ) \sim \mathcal{G}(\frac{1}{f})$,  we call the flow $\overline{\overline{\mathcal{N}}} (f)$ and also the graph $\mathcal{G}(f )$ {\it self-dual}. More general: $\mathcal{G}$ is called self-dual if $\mathcal{G} \!\! \sim \!\! - \mathcal{G}^{*}$.}. Nevertheless, from our point of view it is reasonable to consider the dual flows $\overline{\overline{\mathcal{N}}} (f)$ and $\overline{\overline{\mathcal{N}}} (\frac{1}{f})$ as equal (since the phase portraits are equal, up to the orientation of the trajectories).
   So, the problem of classifying structurally stable elliptic Newton flows is reduced to the classification (under equivalency and duality) of Newton graphs. 

  If $r=2$, the {\it A-property} always holds\footnote{\label{Vtnt7} In $r=2$ we proved \cite{HT2} that all structurally stable $\overline{\overline{\mathcal{N}}} (f)$ are mutually conjugate. So, it is to be
    expected that, in this case, all Newton graphs are equal; see also the forthcoming Remark 2.3} and if $r=3$ the {\it E-property} implies the {\it A-property}, whereas, in case $r=4$,  possibly the {\it A-property} holds, but not the {\it E-property} (cf. \cite{HT2}, Lemma 3.17, Remark 3.18).

  So, to avoid
  a further analysis of the {\it A-property}, we only deal with the cases $r=2, 3$. 
  
  \begin{figure}[h!]
\begin{center}
\includegraphics[scale=0.5]{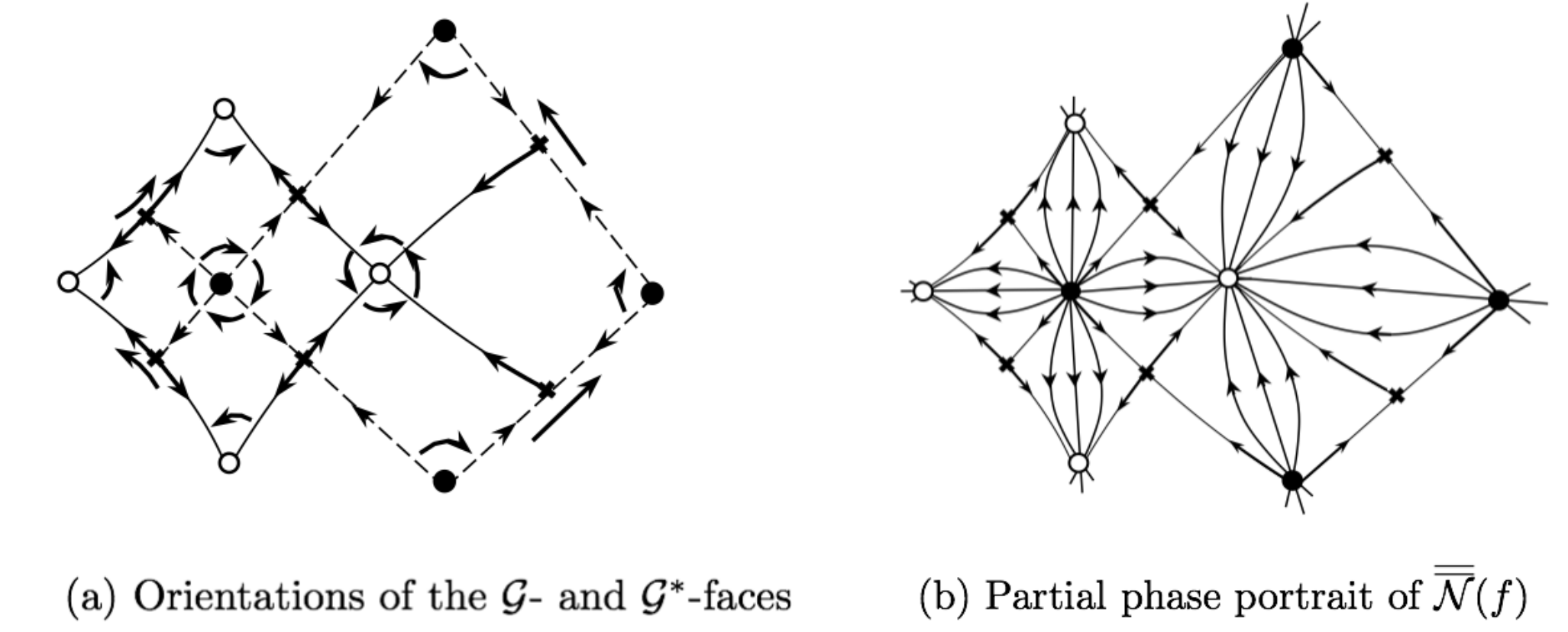}
\caption{\label{Figure2}}
\end{center}
\end{figure}

\noindent
{\large{\bf  {\small 1.2 Some concepts from graph theory}}}\\

We  recall here some concepts from graph theory that are used in the sequel; see \cite{MoTh} as a reference. 
Let $\mathcal{G}$ be a cellularly embedded graph on $T$, not necessarily fulfilling the E-property. We consider the {\it rotation system} $\Pi$ for $\mathcal{G}$:
$$
\Pi=\{ \pi_{v} \!\mid \! \text{all vertices $v$ in $\mathcal{G}$} \},
$$
where the {\it local  rotation system} $\pi_{v}$ at $v$ is the cyclic permutation of the edges incident with $v$ such that $\pi_{v}(e)$ is the successor of $e$ in the anti-clockwise ordering around $v$. The boundary of a face of $\mathcal{G}$ is formally described by the {\it Face traversal procedure}.\\

\noindent
{\bf Face traversal procedure:}\\
If $e(=(v'v''))$ stands for an edge, with end vertices $v'$ and $v''$, 
a $\Pi$-{\it (facial) walk} 
$w$, on $\mathcal{G}$ is defined by: 
Let  $e_{1}=(v_{1}v_{2})$ be an edge. Then the closed walk
$w=v_{1}e_{1}v_{2}e_{2}v_{3} \cdots v_{k}e_{k}v_{1}$ determined by the requirement that, for $i=1, \cdots , \ell,$ we have $\pi_{v_{i+1}}(e_{i})=e_{i+1}$, where $e_{\ell +1}=e_{1}$ and $\ell$ is minimal\footnote{\label{FTN3}Apparently, such "minimal" $l$ exists since $\mathcal{G}$ is finite. In fact, the first edge which is repeated in the same direction when 
traversing $w$, is $e_1$.}, is the desired $\Pi$-(facial) walk.
\\ 

\noindent
Note that each edge occurs either {\it once} in two different $\Pi$-walks, or {\it twice} (with opposite orientations) in only
one $\Pi$-walk. $\mathcal{G}$ has the {\it E-property} iff the first possibility holds for all $\Pi$-walks. The dual $\mathcal{G}^{*}$ admits a
loop iff the second possibility occurs at least in one of the $\Pi$-walks.
Thus, we have:\\

\noindent
Under the {\it E-property} for $\mathcal{G}$, the dual $\mathcal{G}^{*}$ has no loops\footnote{\label{Vtn9}Note that by assumption $\mathcal{G}$ has no loops.} and each $\mathcal{G}$-edge is adjacent to different faces; in fact, any $\mathcal{G}$-edge, say $e$, determines precisely one $\mathcal{G}^{*}$-edge $e^{*}$ (and vice versa) so that
there are $2r$ ÒintersectionsÓ $s=(e, e^{*} )$ of $\mathcal{G}$- and $\mathcal{G}^{*}$-vertices.\\

A crucial principle in our considerations, is \\

\noindent
{\bf The Heffter-Edmonds-Ringel rotation principle:}\\
By this principle, the graph $\mathcal{G}$ is
uniquely determined up to an orientation preserving
isomorphism by its rotation system. 
In fact, consider for each $\Pi$-walk $w$ of length $l$, a so-called $\Pi$-polygon in the plane with $l$ sides labelled by the edges of $w$, so that each polygon is disjoint from the other polygons. These  polygons can be used to construct (patching them along identically labelled sides) an orientable surface $S$ and in $S$ a 2-cell embedded graph $\mathcal{H}$ with faces determined by the polygons. Then $S$ is homeomorphic to $T$ and $\mathcal{H}$ isomorphic with $\mathcal{G}$.

The clockwise oriented  $\Pi$-walks of $\mathcal{G}$ determine a clockwise rotation system $\Pi^{*}$ for $\mathcal{G}^{*}$ that - by the face traversal procedure - leads to anti-clockwise oriented  $\Pi^{*}$-walks for $\mathcal{G}^{*}$.
Occasionally, $\mathcal{G}$  and $\mathcal{G}^{*}$ will be referred to as to the pair ($\mathcal{G}$, $\Pi$) resp. ($\mathcal{G}^{*}$, $\Pi^{*}$).

\section{Classification of Newton graphs of order 3}
\label{sec11} 

Let $\mathcal{G}(=(\mathcal{G}, \Pi))$ be an arbitrary Newton graph 
of order $r$, and $\mathcal{G}^{*}(=(\mathcal{G}^{*}, \Pi^{*}))$ a geometrical dual of $\mathcal{G}$.
The graph
$\mathcal{G}^{*}$ is also a Newton graph of order $r$, see Subsection 1.1. 
The vertices and faces of $\mathcal{G}$ are denoted by $v_{i}$, respectively by $F_{r+i}$. The $\mathcal{G}^{*}$-vertex ``located'' in $F_{r+i}$ is denoted by $v^{*}_{r+i}$, and the $\mathcal{G}^{*}$-face that ``contains'' $v_{i}$ by $F^{*}_{i}, i=1, \! \cdots \! , r.$
In forthcoming figures, the vertices $\mathcal{G}$ and $\mathcal{G}^{*}$ will be indicated by their indices in combination with the symbols $\circ$ and   $\text{\textbullet}$  respectively, 
i.e. $v_{i} \leftrightarrow \circ_{i}$, and $v^{*}_{r+i} \leftrightarrow  \text{\textbullet}_{r+i}$. 
This induces an indexation of the faces of $\mathcal{G}$ and $\mathcal{G}^{*}$ as follows: $F_{r+i}  \leftrightarrow   \text{\textbullet}_{r+i}$ and $F^{*}_{i} \leftrightarrow \circ_{i}$.

The edges of $\mathcal{G}$ and the corresponding edges of $\mathcal{G}^{*}$ are $e_{k}$, resp. $e^{*}_{k}, k=1, \cdots, 2r $ (compare Subsection 1.2).
The degrees of the $\mathcal{G}$- and $\mathcal{G}^{*}$-edges are denoted by $\delta_{i}={\rm deg}(v_{i})$
, resp. $\delta^{*}_{i}={\rm deg}(v^{*}_{r+i})$.  	
Put  $\delta=(\delta(\mathcal{G}))=(\delta_{1}, \cdots ,\delta_{r}), \delta^{*}=(\delta(\mathcal{G}^{*}))=(\delta^{*}_{1}, \cdots ,\delta^{*}_{r})$ and note that there holds
$\delta(\mathcal{G}^{*})=\delta^{*}(\mathcal{G})$ and $\delta^{*}(\mathcal{G}^{*})=\delta(\mathcal{G})$.

We consider the ``common refinement'' $\mathcal{G} \wedge \mathcal{G}^{*}$ of $\mathcal{G}$ and $\mathcal{G}^{*}$. This graph\footnote{\label{FTN7}The abstract, directed graph underlying $\mathcal{G} \wedge \mathcal{G}^{*}$ is denoted by $\mathbb{P} (\mathcal{G})$, see also the forthcoming Fig.\ref{Figure9.6} (ii), Fig.\ref{Figure9.9}(i). For the significance of $\mathbb{P} (\mathcal{G})$ within the theory of structurally stable dynamical systems, see the papers \cite{HT2}, \cite{Peix2}.} is defined by: It has vertices on three levels:

\noindent
{\it Level } 1:  The vertices $v_{i}$ of $\mathcal{G}$,

\noindent
{\it Level } 2:  The ``intersections'' $s_{k}$ of the pairs $(e_{k},e^{*}_{k})$, compare Subsection 1.2,

\noindent
{\it Level } 3:  The  vertices $v^{*}_{r+i}$ of $\mathcal{G}^{*}$,\\
\noindent
whereas each $\mathcal{G}$-edge $e_{k}$ (each $\mathcal{G}^{*}$-edge $e^{*}_{k}$) is partitioned into two $\mathcal{G} \wedge \mathcal{G}^{*}$-edges connecting  
$s_{k}=(e_{k},e^{*}_{k})$ with the end vertices of $e_{k}$ (of $e^{*}_{k}$). Moreover, there are no $\mathcal{G} \wedge \mathcal{G}^{*}$-connections
between vertices on Level 1 and 3.
\begin{lemma}
\label{NL11.2}
The following relations hold:
\begin{equation*}
1 < \delta_{i} \leqslant 2r, 1 < \delta^{*}_{i} \leqslant 2r, \sum_{i=1}^{r} \delta_{i}  = \sum_{i=1}^{r}   \delta^{*}_{i}=4r.
\end{equation*}
By construction, $\mathcal{G} \wedge \mathcal{G}^{*}$ has precisely 4$r$ faces, and moreover, 
\begin{equation*}
\text{no } s_{k } \text{ is connected by two $\mathcal{G} \wedge \mathcal{G}^{*}$-edges to the same }v_{i} \text{ or the same }v^{*}_{r+i}.
\end{equation*}
\end{lemma}
\begin{proof}
Since both $\mathcal{G}$ and $\mathcal{G}^{*}$ are Newtonian (Subsection 1.1)
, it follows by the {\it E-property} that these graphs do not admit loops (cf. Subsection 1.2)
, whereas the {\it A-property} ensures the non-existence of vertices for $\mathcal{G}$ and $\mathcal{G}^{*}$ of degree 1.
\end{proof}

From now on, let $\mathcal{G}$ be a $3^{rd}$ order Newton graph.
We adapt the notations: the $\mathcal{G}$-edges will be denoted by $a, b, c, d, e, f$, and the corresponding $\mathcal{G}^{*}$-edges by 
$a^{*}, b^{*}, c^{*}, d^{*}, e^{*}, f^{*}$. The $\mathcal{G} \wedge \mathcal{G}^{*}$-vertices determined by $(a,a^{*}), (b,b^{*}), \cdots$ are denoted by respectively $a, b, \cdots$.\\

\noindent
Our aim is  
a complete classification (up to 
equivalency) of all graphs $\mathcal{G}$, where we use
that, since $r=3$, 
the {\it E-property} already implies that  $\mathcal{G}$ is a Newton graph. (cf. Subsection 1.1)\\

We distinguish between the following three possibilities with respect to the boundaries (or $\Pi$-walks) of $\mathcal{G}$-faces 
:\\

\noindent
\underline{Case 1}:  The boundary of one of the $\mathcal{G}$-faces, say $\partial F_{4}$, has six edges, i.e. $\delta^{*}_{4}=6$ .

\noindent
\underline{Case 2}:  The boundary of one of the $\mathcal{G}$-faces, say $\partial F_{4}$, has five edges, i.e. $\delta^{*}_{4}=5$.

\noindent
\underline{Case 3}:  {\bf Each} boundary of the faces in $\mathcal{G}$ and $\mathcal{G}^{*}$ has four 
edges, i.e. 
$
\delta=\delta^{*}=(4,4,4).
$
\\

By Lemma \ref{NL11.2} the Cases 1, 2 and 3 are mutually exclusive and cover all possibilities. First we should check wether there exist graphs $\mathcal{G}$ that fulfil the conditions in the above cases, and, even so, to what extent $\mathcal{G}$ is determined by these conditions.\\

\noindent
{\it Ad Case 1}: Because of the {\it E-property}, and since $\mathcal{G}$ has {\it no loops}, it is necessary for the existence of $\mathcal{G}$ that the $\Pi$-walk $w_{F_{4}}$ of a possible face $F_{4}$ fulfils one of the following conditions:\\

\noindent
\underline{Subcase 1.1} :  Traversing 
$w_{F_{4}}$ once, each vertex appears precisely twice.

\noindent
\underline{Subcase 1.2} :  Traversing $w_{F_{4}}$
once, there is one vertex (say $v_{1}$) appearing three times, one (say $v_{2}$) appearing twice, and one (say $v_{3}$) showing up only once.\\

\noindent
The ({\it clockwise oriented})``$\Pi$-polygon'' for $\partial F_{4}$ has six sides, labelled
$a, b, \cdots, f$ and six ``corner points'', labelled by the vertices $v_{1},v_{2},v_{3}$ (repetitions necessary). 
Identifying points related to the same $\mathcal{G}$-vertex, brings us back to $w_{F_{4}}$. Assume that the cyclic permutations of the edges in $w_{F_{4}}$ that are incident with the same vertex are oriented {\it anti-clockwise} (compare the conventions in Subsection 1.1)

In Subcase 1.1 there are precisely two different -up to relabeling- possibilities for $w_{F_{4}}$ according to the schemes: (see Fig. \ref{Figure9.1})
\begin{equation}
\label{1-1-1}
w_{F_{4}}: v_{1}av_{2}bv_{3}cv_{1}dv_{2}ev_{3}f v_{1}av_{2}
\end{equation}
or
\begin{equation}
\label{1-1-2}
w_{F_{4}}: v_{1}av_{2}bv_{3}cv_{2}d v_{1}ev_{3}fv_{1}av_{2}.
\end{equation}

\begin{figure}[h!]
\begin{center}
\includegraphics[scale=0.55]{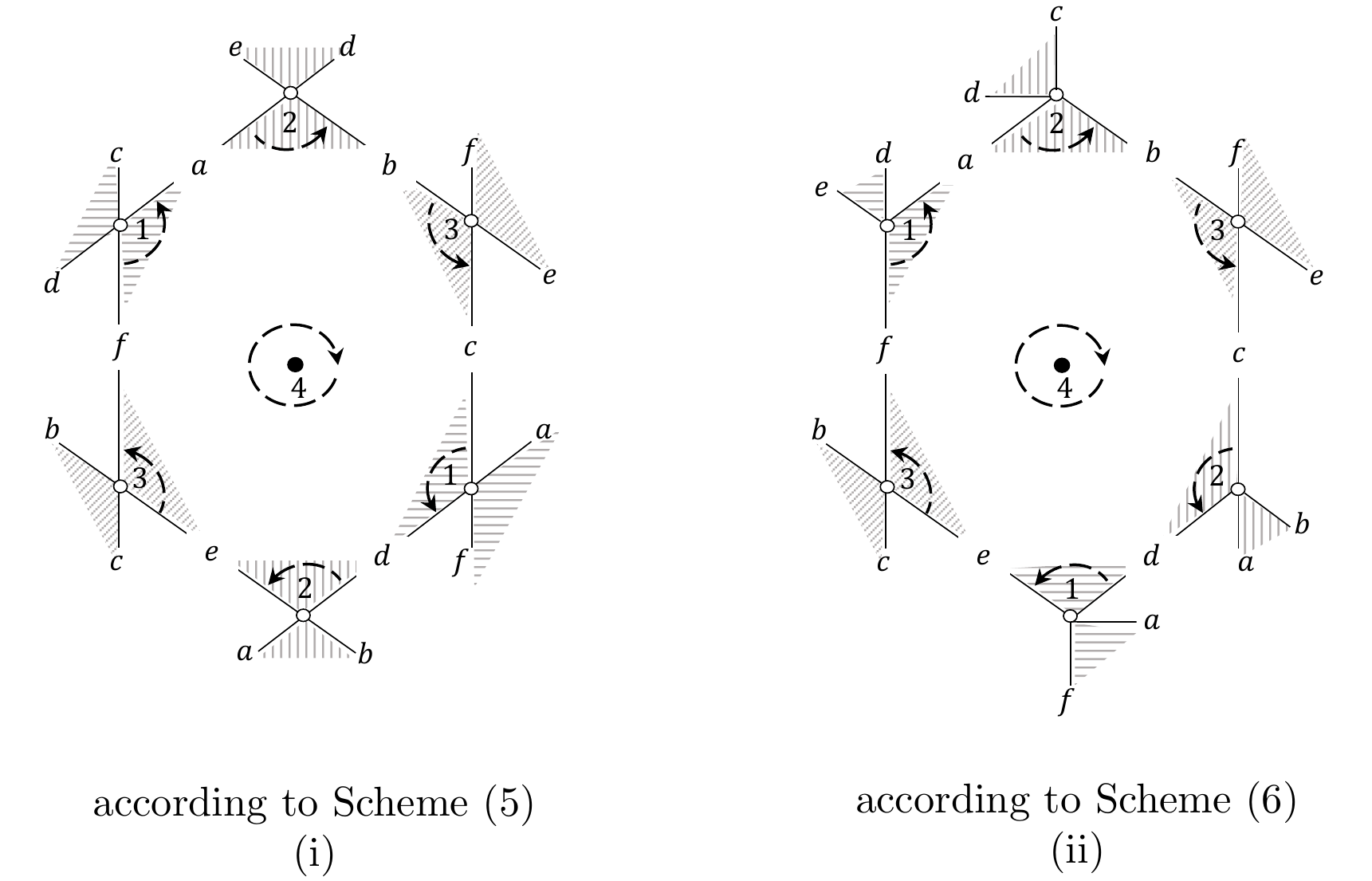}
\caption{\label{Figure9.1} The two possibilities for $w_{F_{4}}$ in Subcase 1.1.}
\end{center}
\end{figure}

First, we focus on $w_{F_{4}}$ given by Scheme (\ref{1-1-1}), see Fig. \ref{Figure9.1}(i). In the ({\it anti-clockwise}) cyclic permutation of the $w_{F_{4}}$-edges, incident with the same vertex, these edges occur in pairs,  determining a (positively oriented) sector of $F_{4}$.
As an edge is always adjacent to two different faces (cf. Subsection 1.2),
two $F_{4}$-sectors at the same $v_{i}$ are separated by facial sectors (at $v_{i}$) {\it not} belonging to $F_{4}$ (cf. Fig. \ref{Figure9.1}(i)). Since, moreover, the graph we are looking for, admits altogether twelve facial sectors, the cyclic permutation of the edges at $v_{i}$ are as indicated in Fig. \ref{Figure9.1}(i) and constitute a rotation system that -upto equivalency and relabeling- determines the graph, say $\mathcal{G}$, uniquely.

With the aid of the rotation system in Fig. \ref{Figure9.1}(i) and applying the {\it face traversal procedure}, as sketched in Subsection 1.2
, we find the closed walks $v_{2}av_{1}cv_{3}ev_{2}av_{1}$ 
and $v_{2}dv_{1}fv_{3}bv_{2}dv_{1}$ 
defining the two other $\mathcal{G}$-faces, say $F_{5}$, resp. $F_{6}$. (Note that each edge occurs twice in different walks, but with opposite orientation).
Glueing together the facial
polygons corresponding to $F_{4}$, $F_{5}$ and $F_{6}$, according to equally labeled sides and corner points, gives rise to the plane representations of $\mathcal{G}$ in 
Fig.\ref{Figure9.2}-(i).

\begin{figure}[h!]
\begin{center}
\includegraphics[scale=0.25]{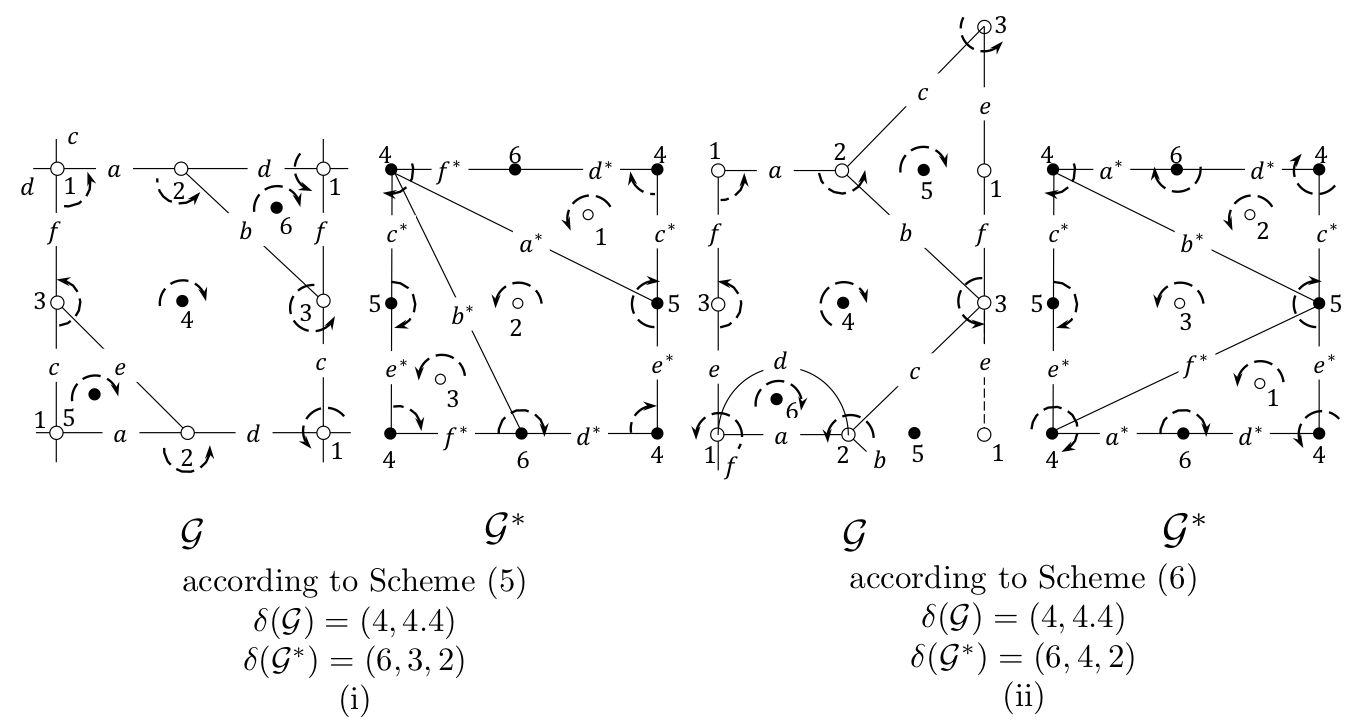}
\caption{\label{Figure9.2} The two possible plane representations for $\mathcal{G}$, $\mathcal{G}^{*}$ in Subcase 1.1.}
\end{center}
\end{figure}

From Fig.\ref{Figure9.2}-(i) it follows that the rotation system for $\mathcal{G}^{*}$ is as depicted in Fig. \ref{Figure9.3}. With the aid of this figure we find, again by the {\it face traversal procedure}, the following closed subwalks in $\mathcal{G}^{*}$:
$v_{4}^{*}a^{*}v_{5}^{*}c^{*}v_{4}^{*}d^{*}v_{6}^{*}f^{*}v_{4}^{*}a^{*}$, 
$v_{4}^{*}b^{*}v_{6}^{*}d^{*}v_{4}^{*}e^{*}v_{5}^{*}a^{*}v_{4}^{*}b^{*}$ and
$v_{4}^{*}f^{*}v_{6}^{*}b^{*}v_{4}^{*}c^{*}v_{5}^{*}e^{*}v_{4}^{*}f^{*}$,
defining the $\mathcal{G}^{*}$-faces $F_{1}^{*}$, $F_{2}^{*}$, $F_{3}^{*}$ respectively. (Note that each edge occurs twice in different walks, but with opposite orientation). Glueing together the facial 
polygons corresponding to these faces
according to equally labeled sides and corner points, yields the plane representations of $\mathcal{G}^{*}$ in 
Fig.\ref{Figure9.2}-(i).

\begin{figure}[h!]
\begin{center}
\includegraphics[scale=0.7]{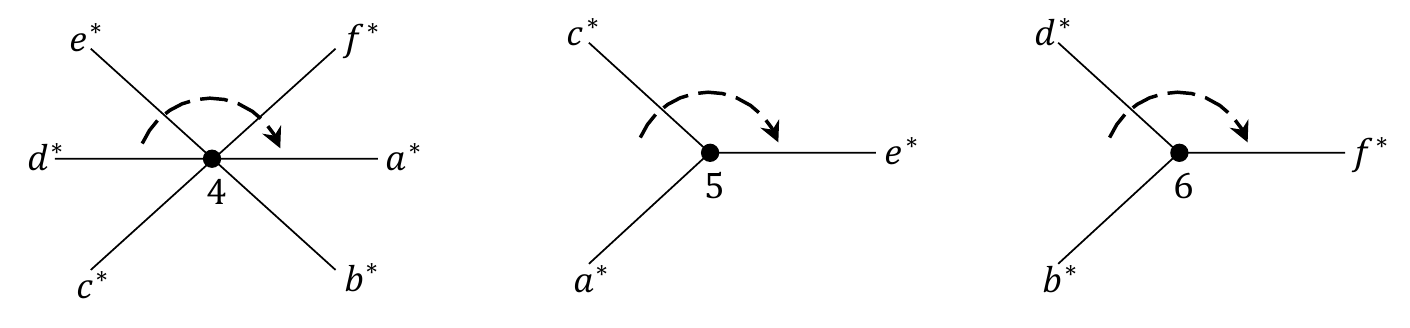}
\caption{\label{Figure9.3} The rotation systems for $\mathcal{G}^{*}$, according to Scheme (\ref{1-1-1}).}
\end{center}
\end{figure}

If we start from a $\Pi$-walk for $F_{4}$, according to the Scheme (\ref{1-1-2}), we find (by the same argumentation as above) plane representations for 
$\mathcal{G}$ and $\mathcal{G}^{*}$; see 
Fig.\ref{Figure9.2}-(ii).\\

\noindent
Note that in all graphs in Fig.\ref{Figure9.2} the anti-clockwise (clockwise) orientation of the cyclic permutations of edges incident with the same vertex induces a clockwise (anti-clockwise) orientation of the faces.\\

\noindent
In Subcase 1.2 there is {\it precisely one} -up to relabeling- possibility for $w_{F_{4}}$ according to the scheme: 

\begin{equation}
\label{1.2}
w_{F_{4}}: v_{1}av_{2}bv_{1}cv_{2}dv_{1}ev_{3}fv_{1}a.
\end{equation}
In this case however, there are {\it three} pairs of $\mathcal{G}$-edges at $v_{1}$ determining (positively measured) sectors of $F_{4}$. 
So, reasoning as in Subcase 1.1, there are two possibilities for the (anti clockwise) cyclic permutations of the $\mathcal{G}$-edges at $v_{1}$ (and thus also two different rotation systems; see Fig. \ref{Figure9.4}).

\begin{figure}[h!]
\begin{center}
\includegraphics[scale=0.5]{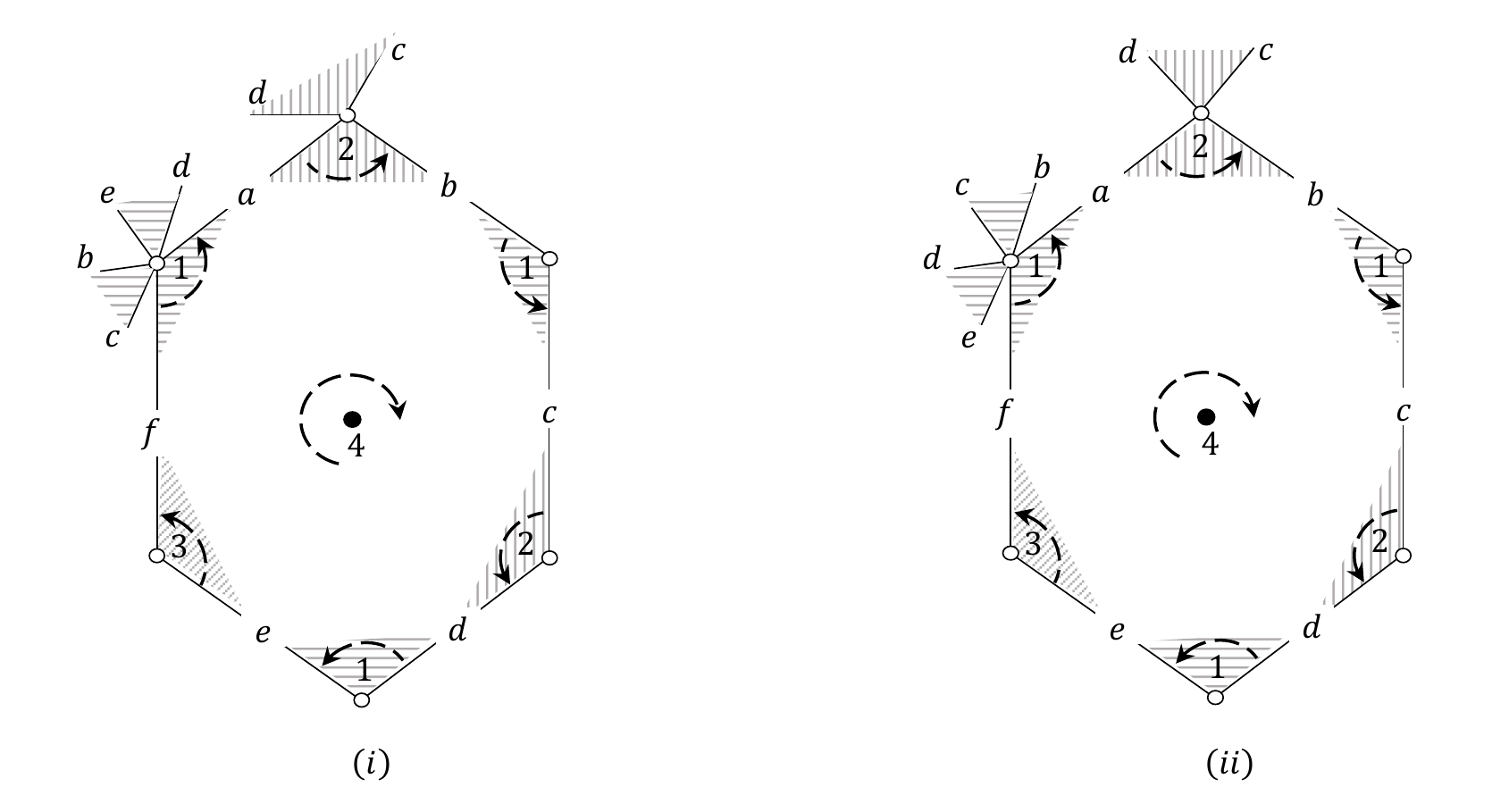}
\caption{\label{Figure9.4} The two possible rotation systems
in Subcase 1.2.}
\end{center}
\end{figure}

Starting from Fig. \ref{Figure9.4}-(i) and applying the {\it face traversal procedure}, we find the facial walks 
$v_{1}fv_{3}ev_{1}bv_{2}cv_{1}f$
and $v_{1}av_{2}dv_{1}a$
, which together with Scheme (\ref{1.2}) define the faces $F_{5}$, $F_{6}$ and $F_{4}$ respectively. Reasoning as in Subcase 1.1, we arrive at the plane realizations of $\mathcal{G}$ and $\mathcal{G}^{*}$ as depicted in Fig.\ref{Figure9.5}-(i). In the case of Fig.\ref{Figure9.4}-(ii) the facial walks $v_{1}dv_{2}av_{1}bv_{2}cv_{1}d$ 
and $v_{1}ev_{3}fv_{1}e$,
together with Scheme (\ref{1.2}), define the faces $F_{5}$, $F_{6}$ and $F_{4}$ respectively. 
Reasoning as in Subcase 1.1, we obtain
the plane representations for $\mathcal{G}$ and $\mathcal{G}^{*}$ as depicted in Fig.\ref{Figure9.5}-(ii).

Note that
both graphs $\mathcal{G}$ in Fig.\ref{Figure9.5} are {\it self dual} (cf. footnote \ref{Vtnt6}, or note that $\delta(\mathcal{G})= \delta(\mathcal{G}^{*})$ and use Lemma 3.5 in \cite{HT2}), but-by inspection of their rotation systems-are {\it not equivalent} (cf.  Subsection 1.2).\\

\begin{figure}[h!]
\begin{center}
\includegraphics[scale=0.2]{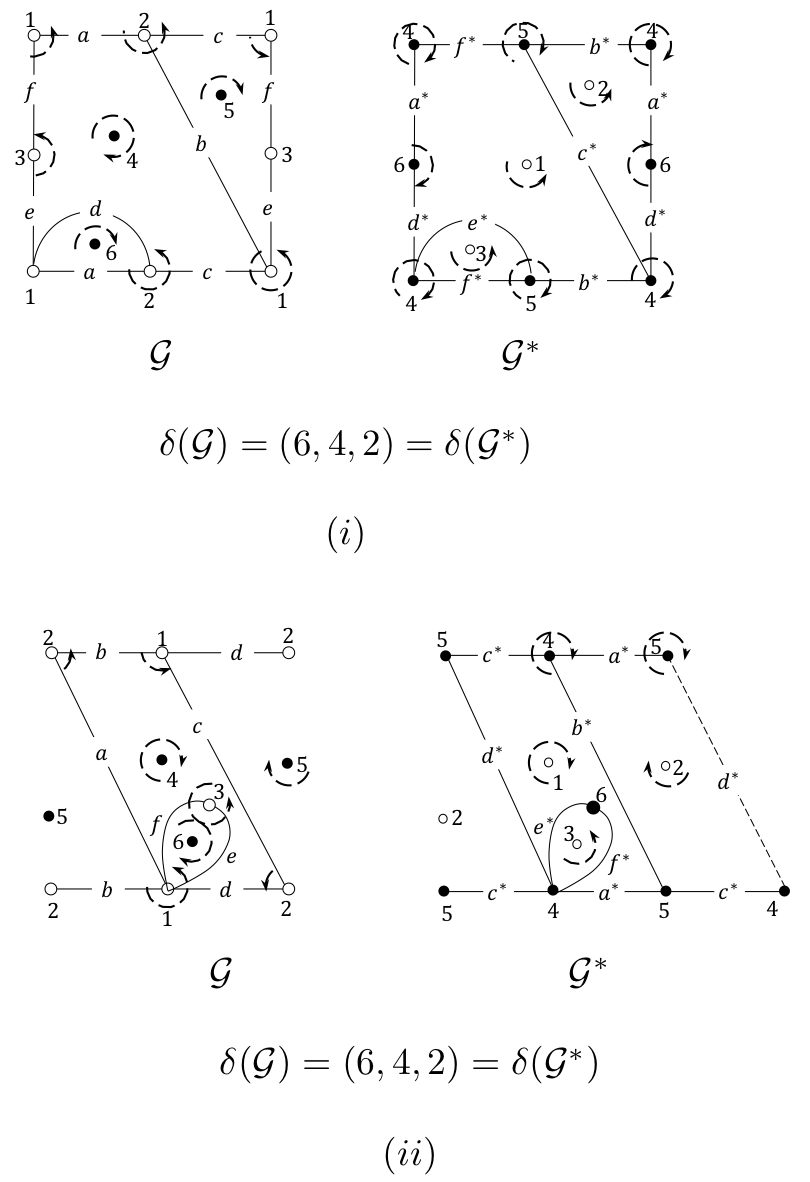}
\caption{\label{Figure9.5} The two possible plane representations for $\mathcal{G}$, $\mathcal{G}^{*}$ in Subcase 1.2.}
\end{center}
\end{figure}

\noindent
{\it Ad Case 2}: Because the $\Pi$-walk of $F_{4}$ has no loops and consists of an Euler trail on the five edges of $\mathcal{G}$, there is only one- up to relabeling - possibility for $w_{F_{4}}$ (see Fig.\ref{Figure9.6}-(i)):
$$
w_{F_{4}}: v_{1}av_{3}bv_{2}cv_{1}dv_{2}ev_{1}a.
$$
In contradistinction with the previous Case 1, now there is one edge, namely $f$, that is not contained in $w_{F_{4}}$. By Lemma \ref{NL11.2}, this edge must connect either $v_{1}$ to $v_{2}$ ($f: v_{1} \leftrightarrow v_{2}$), or $v_{1}$ to $v_{3}$ ($f: v_{1} \leftrightarrow v_{3}$), or 
$v_{2}$ to $v_{3}$ ($f: v_{2} \leftrightarrow v_{3}$); compare Fig. \ref{Figure9.6}-(ii) where we show the part of the abstract graph $\mathbb{P} (\mathcal{G})$ underlying $\mathcal{G} \wedge \mathcal{G}^{*}$ that is determined by $\partial F_{4}$. To begin with, we focus on the first two sub cases.

Taking into account the various positions of $f$ with respect to local sectors of $F_{4}$ at $v_{1}$ and $v_{2}$ (when $f: v_{1} \leftrightarrow v_{2}$), respectively $v_{1}$ and $v_{3}$ (when $f: v_{1} \leftrightarrow v_{3}$), we find {\it four} respectively {\it two} possibilities for the rotation systems; see Fig.\ref{Figure9.7}. The Subcases 
$f: v_{1} \leftrightarrow v_{3}$ and $f: v_{2} \leftrightarrow v_{3}$ are not basically different\footnote{\label{FN22} Relabeling $v_{1} \leftrightarrow v_{2}$, $a \leftrightarrow b$ and $c \leftrightarrow e$, transforms the two configurations in Fig.\ref{Figure9.7}(v) and (vi) into configurations that generate (anti-clockwise oriented) rotation systems describing the case $f: v_{2} \leftrightarrow v_{3}$.}.
So, we may neglect the case $f: v_{2} \leftrightarrow v_{3}$. Reasoning as in Case 1, the rotation systems in Fig.\ref{Figure9.7} yield 
the possible planar representations of $\mathcal{G}$ and $\mathcal{G}^{*}$; see Fig.\ref{Figure9.8}.
Note that- by inspection of their rotation systems -all graphs $\mathcal{G}$ in this figure are different under {\it orientation preserving isomorphisms}, whereas {\it only} in the cases of Fig.\ref{Figure9.8}(iii) and (iv) these graphs are equal w.r.t. an {\it orientation reversing isomorphism} (apply the relabeling introduced in Footnote \ref{FN22}). Apparently, the graphs $\mathcal{G}$ and $\mathcal{G}^{*}$ (and thus also $\mathcal{G}^{*}$ and $\mathcal{G}$) in Fig.\ref{Figure9.8}(i), resp. Fig.\ref{Figure9.8}(v) are equal (under an orientation preserving isomorphism) 
The graphs 
$\mathcal{G}$ in Fig. \ref{Figure9.8} (ii),(iii),(iv),(vi) are {\it self-dual}.

\begin{figure}[h!]
\begin{center}
\includegraphics[scale=0.5]{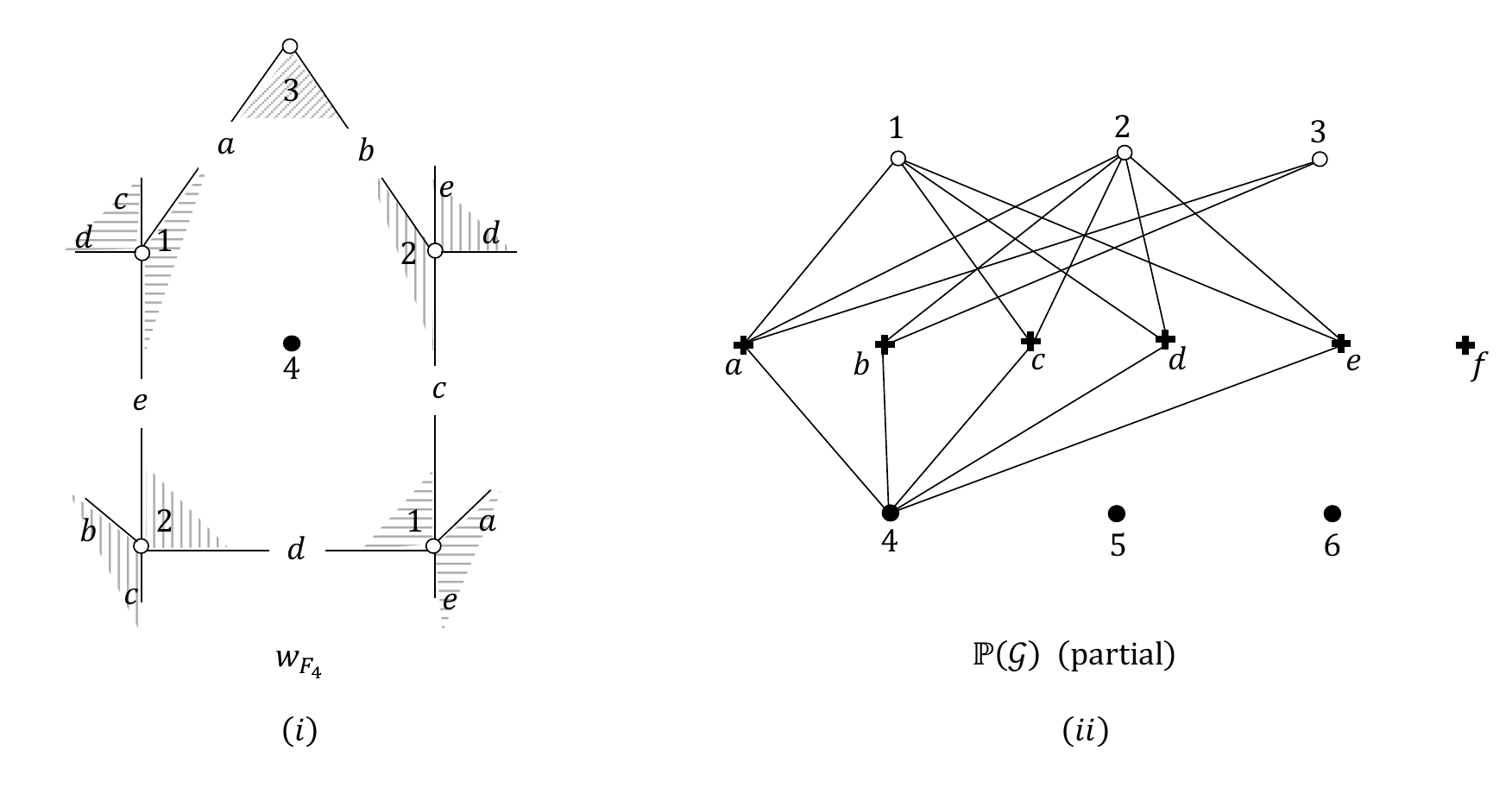}
\caption{\label{Figure9.6} The $\Pi$-walk for $F_{4}$ in Case 2.}
\end{center}
\end{figure}

\begin{figure}[h!]
\begin{center}
\includegraphics[scale=0.5]{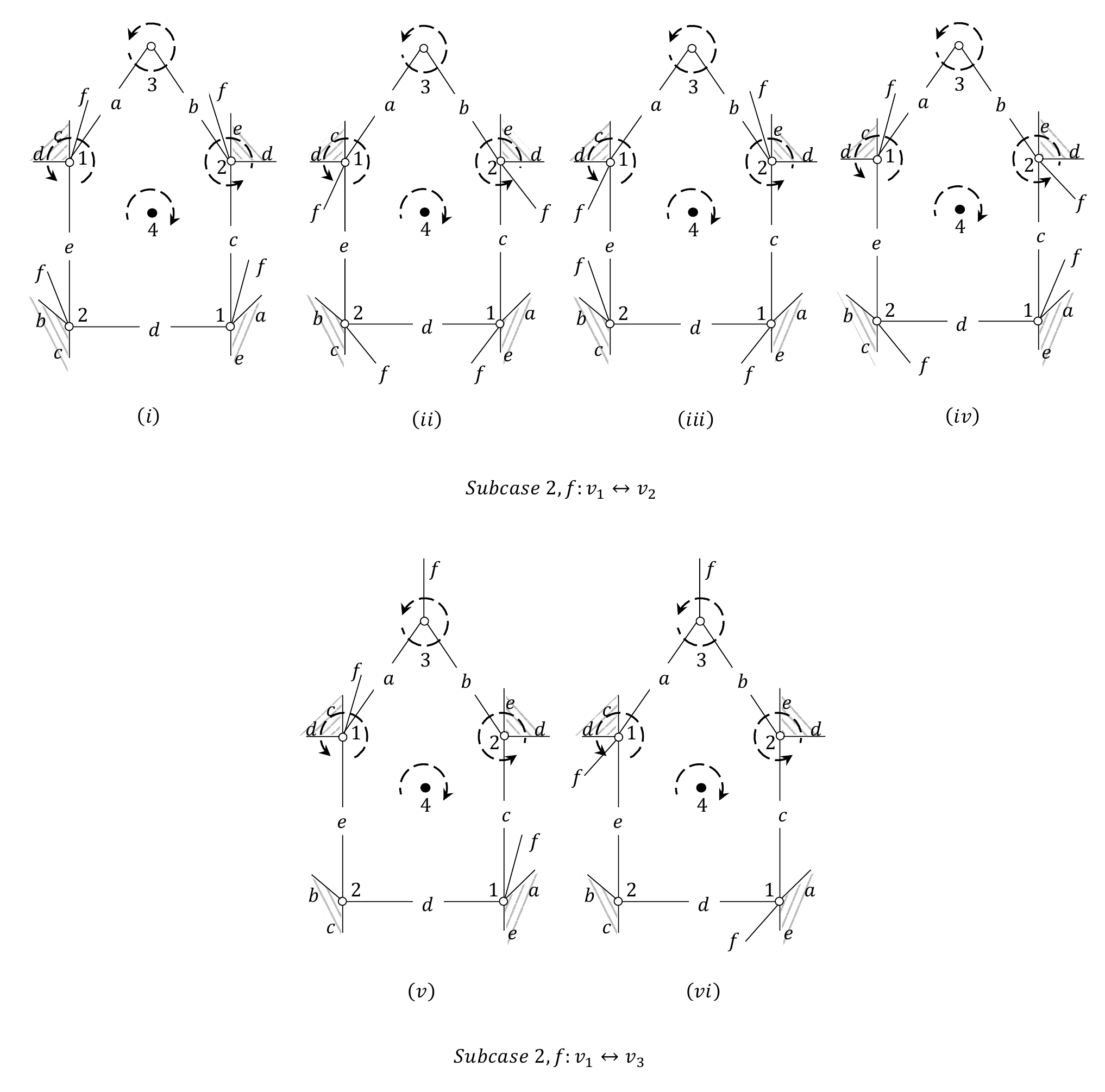}
\caption{\label{Figure9.7} The possible rotation systems for $\mathcal{G}$ in Case 2.}
\end{center}
\end{figure}

\begin{figure}[h!]
\begin{center}
\includegraphics[scale=0.25]{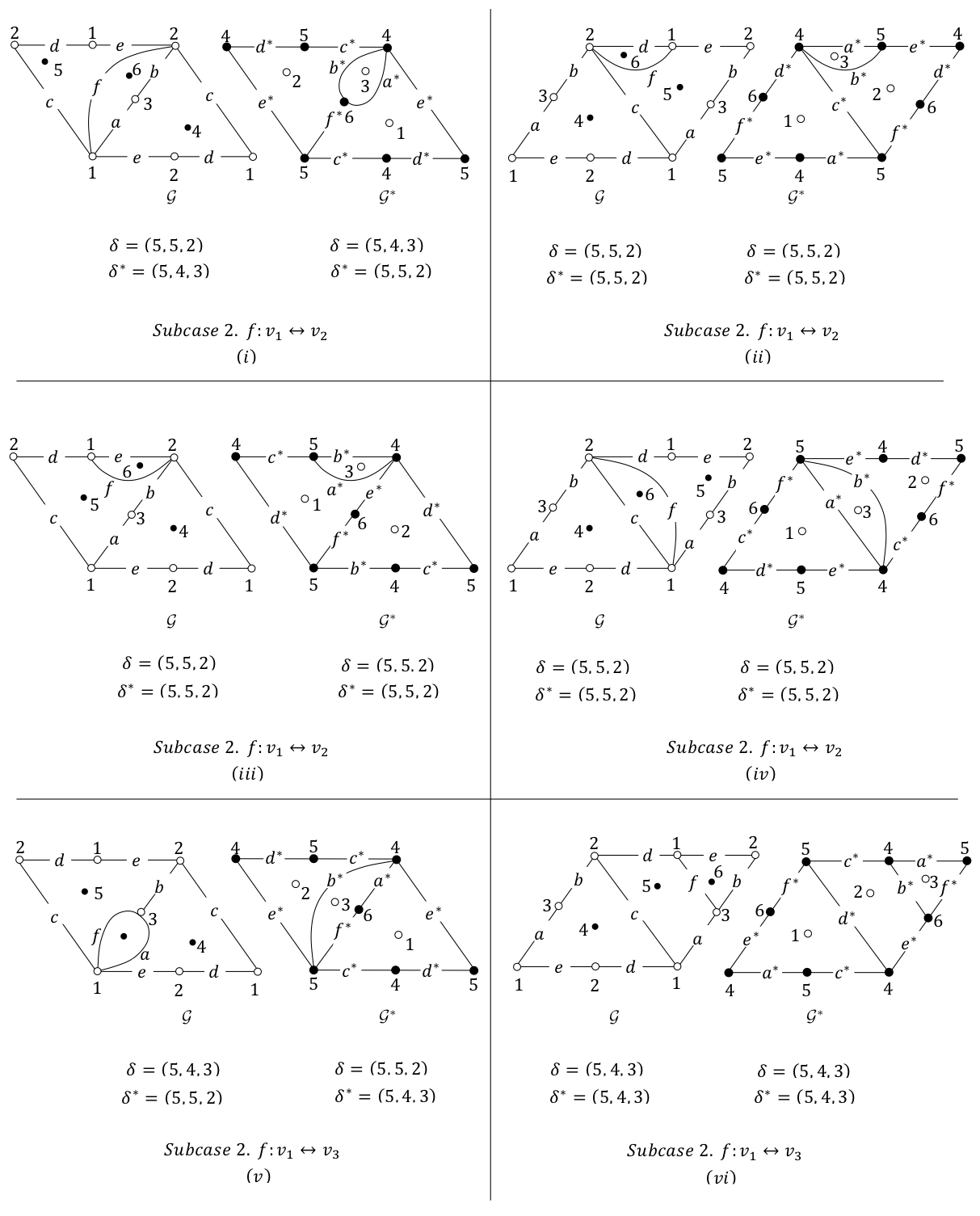}
\caption{\label{Figure9.8} The graphs $\mathcal{G}$ and $\mathcal{G}^{*}$ in Case 2.}
\end{center}
\end{figure}

\noindent
{\it Ad Case 3}: Without loss of generality, there are a priori {\it two} possibilities for the $\Pi$-walks of an arbitrary face, say $F_{4}$; see 
Fig.\ref{Figure9.9}(i) and (ii). By Lemma \ref{NL11.2} and by inspection of the corresponding partial graph $\mathbb{P} (\mathcal{G})$, the first possibility is ruled out. So, we focus on Fig. \ref{Figure9.9}(ii). Recall that two facial sectors at the same vertex $v_{i}$ are separated by facial sectors (at $v_{i}$) {\it not} belonging to $F_{4}$ and that in the actual case we have $\delta\!=\!\delta^{*}\!=\!(4,4,4)$. So, we find the rotation systems and the distribution of ``local facial sectors'' as depicted in Fig. \ref{Figure9.9}(ii), where the roles of both $e$, $f$ and $F_{5}$, $F_{6}$ may be interchanged. Now, by the face traversal procedure we find:
$$
\text{Apart from relabeling and equivalency, there is only one (self dual) graph possible,
Fig. \ref{Figure9.10}}.
$$

\newpage

\begin{figure}[h!]
\begin{center}
\includegraphics[scale=0.65]{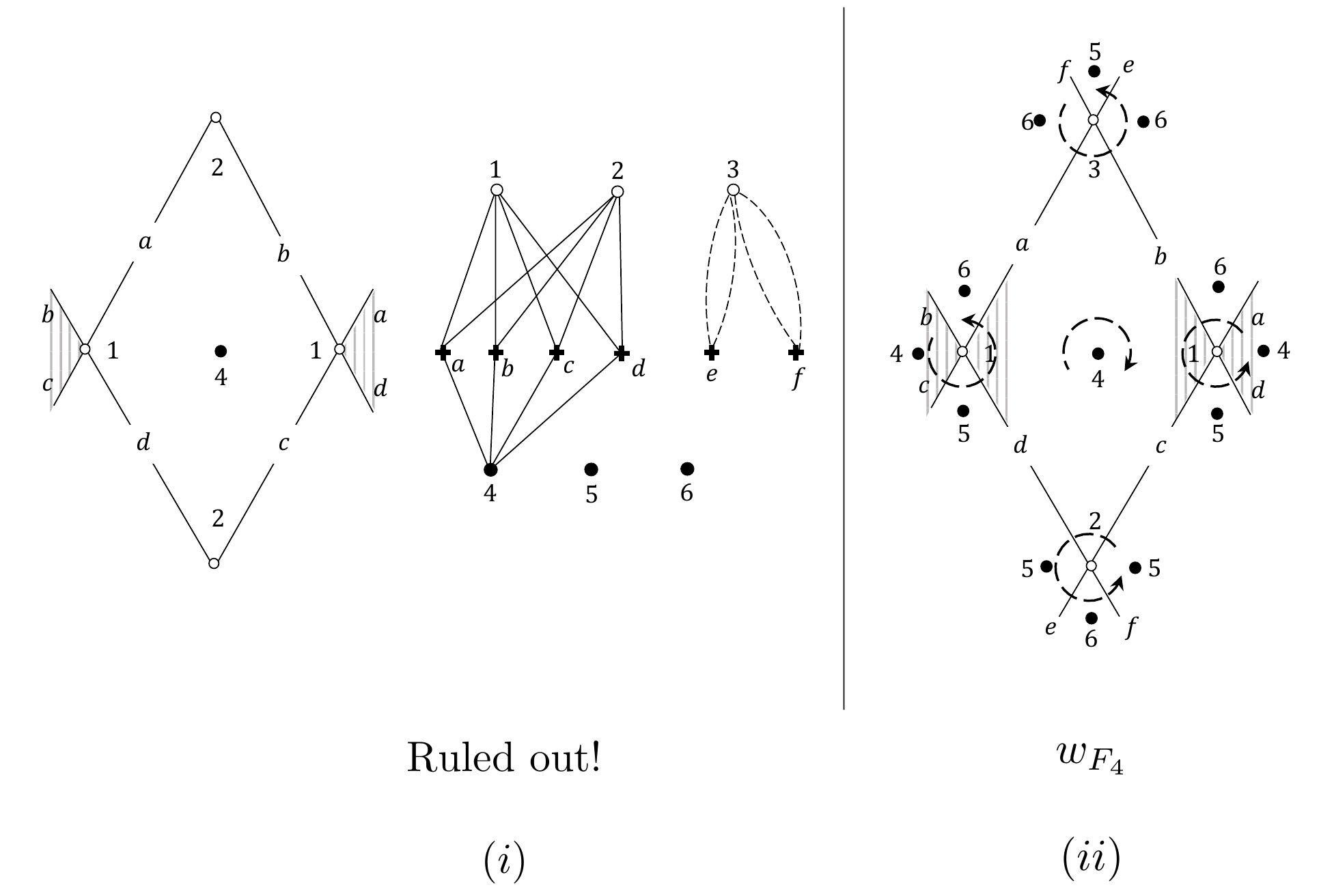}
\caption{\label{Figure9.9} Apriori possibilities for $w_{F_{4}}$ in Case 3.}
\end{center}
\end{figure}

\begin{figure}[h!]
\begin{center}
\includegraphics[scale=0.6]{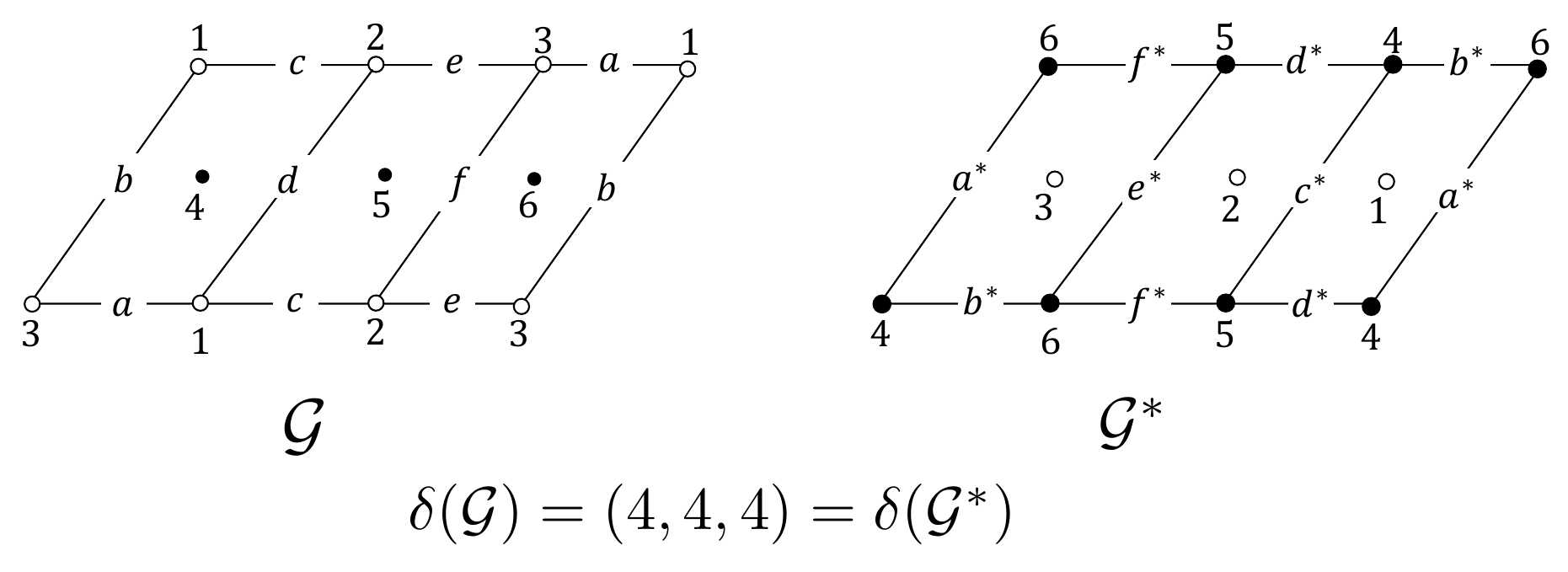}
\caption{\label{Figure9.10} The only possible graphs $\mathcal{G}$ and $\mathcal{G}^{*}$ in Case 3.}
\end{center}
\end{figure}

\newpage

\noindent
Now the representation result from Subsection 1.1
and the remark there about dual flows
, together with the above analysis of the $3^{rd}$ order Newton graphs yields:

\begin{theorem}$($Classification of third order Newton graphs$)$
\label{T11.3} 
\begin{itemize}
\item Apart from conjugacy and duality, there are precisely {\it nine} possibilities for the $3^{rd}$ order structurally stable elliptic Newton flows. These possibilities are characterized by the Newton graphs in Fig.\ref{Figure9.11}.
\item If we add to Fig.\ref{Figure9.11} the duals of the graphs in Fig.\ref{Figure9.11} $($i$)$, $($ii$)$, $($v$)$, we obtain a classification under merely conjugacy, containing {\it twelve} different possibilities.
\end{itemize}
\end{theorem}

\begin{remark}
\label{R7.4} {\it The Case $r=2$. }\\
By similar (even easier) arguments as used in the above Case $r=3$, it can be proved that -up to equivalency-there is only one (self-dual) possibility for the $2^{nd}$ order Newton graphs; see Fig.\ref{Figure9.12} (Note that in view of the {\it E-property} both facial walks of such graphs have length 4, whereas
the role of the {\it A-property} is not relevant, see Subsection 1.1). For a different approach, see Corollary 2.13 in \cite{HT2}. 
\end{remark}

\begin{remark}
\label{R2.4}
In case of degenerate elliptic functions, it is possible to describe the corresponding Newton flows by so-called pseudo Newton graphs, see our 
paper \cite{HT4}.
\end{remark}

\begin{figure}[h!]
\begin{center}
\includegraphics[scale=0.6]{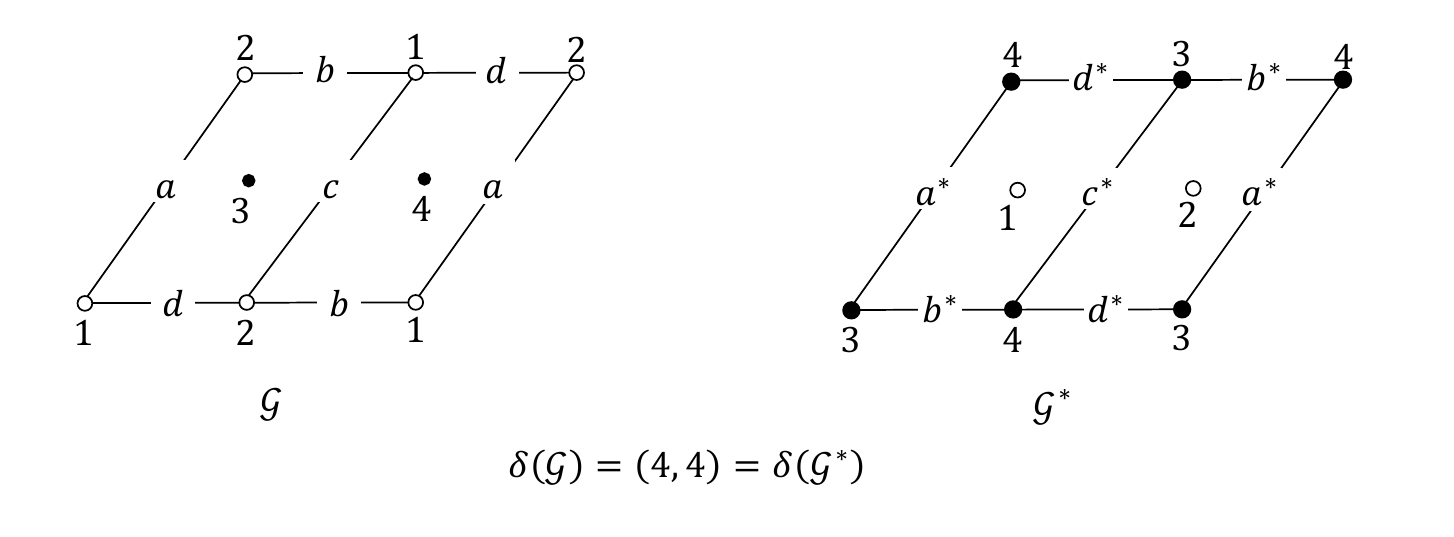}
\caption{\label{Figure9.12} The  $2^{nd}$ order Newton graphs.}
\end{center}
\end{figure}

\newpage

\begin{figure}[h!]
\begin{center}
\includegraphics[scale=0.55]{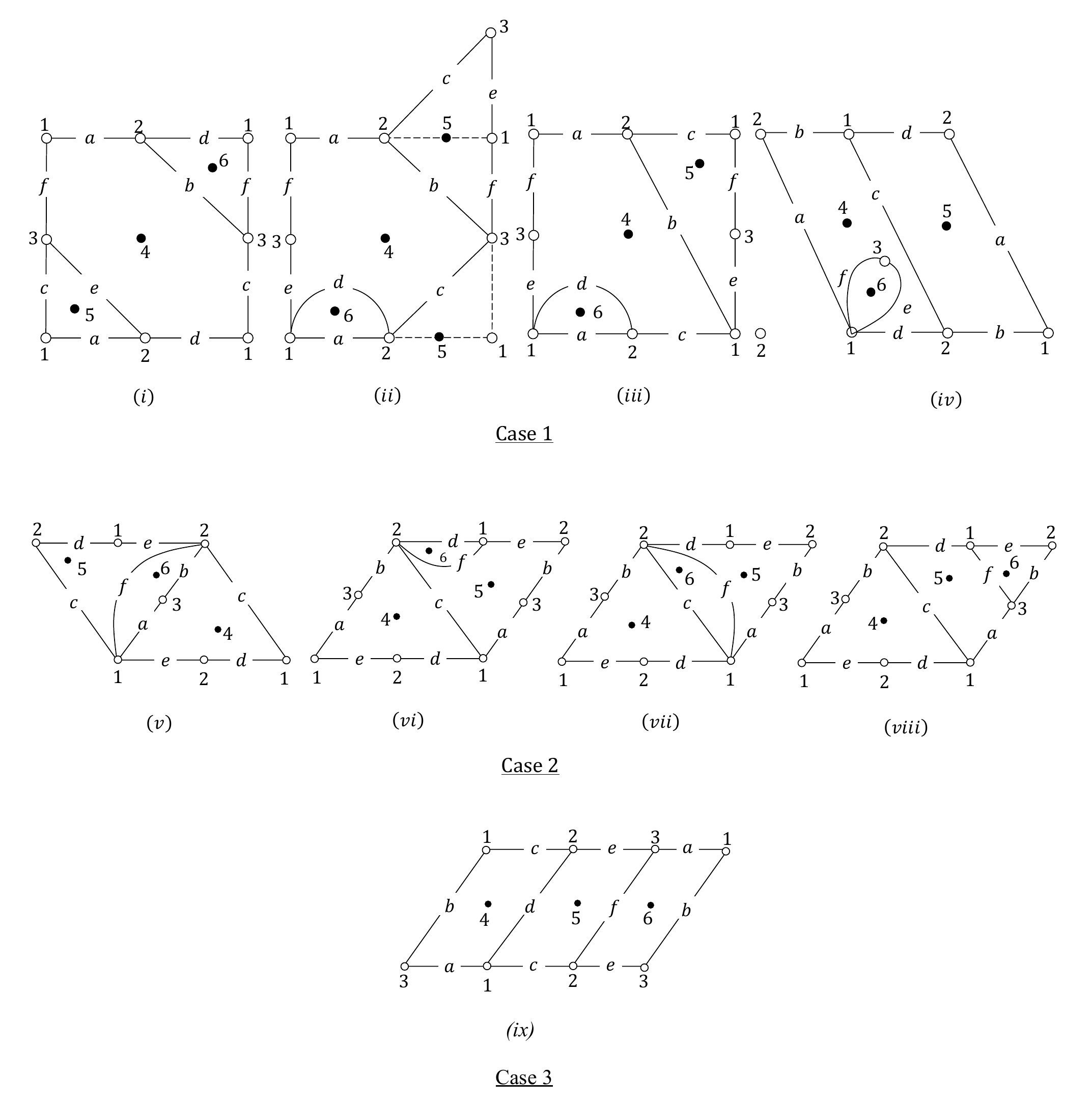}
\caption{\label{Figure9.11} The graphs characterizing structurally stable elliptic Newton flows of order 3.}
\end{center}
\end{figure}

\end{document}